 \def\misajour{03/06/2023}

\documentclass[12pt]{article} 
 
\usepackage[applemac]{inputenc}

\usepackage{amsmath,amsthm}
\usepackage{amssymb} 

\usepackage{enumitem}
 
\usepackage{graphicx}

\usepackage[T1]{fontenc} 

\usepackage{hyperref}   
\usepackage{hyperref} 
\usepackage{xcolor}
	\hypersetup{
    colorlinks,
    linkcolor={red!50!black},
    citecolor={blue!50!black},
    urlcolor={blue!80!black}
}

\newtheorem{theorem}{Theorem}[section]
\newtheorem{corollary}[theorem]{Corollary}
\newtheorem{lemma}[theorem]{Lemma}
\newtheorem{proposition}[theorem]{Proposition}
\newtheorem{conjecture}[theorem]{Conjecture}

\theoremstyle{definition}
\newtheorem{definition}[theorem]{Definition}
\newtheorem{remark}[theorem]{Remark}
\newtheorem{example}[theorem]{Example}

\numberwithin{equation}{section} 
  
\newcommand {\Q}{{\mathbb{Q}}}

\newcommand {\Z}{{\mathbb{Z}}}  

\newcommand{\rmH}{\mathrm {H}}
\newcommand{\rmh}{\mathrm {h}}
\newcommand{\rme}{\mathrm {e}}
\newcommand{\calA}{\mathcal {A}}
\newcommand{\calE}{\mathcal {E}} 
\newcommand{\calF}{\mathcal {F}}
\newcommand{\calG}{\mathcal {G}}
\newcommand{\calR}{\mathcal {R}}
\newcommand{\calZ}{\mathcal {Z}}

\begin{document}


\baselineskip=17pt


\title{Number of integers represented by families 
\\  of binary forms  II: binomial forms}

\author{\'Etienne Fouvry\\
  Universit\' e Paris--Saclay, CNRS,    \\
Laboratoire  de  Math\' ematiques  d'Orsay,     \\
91405   Orsay,   France \\
E-mail: Etienne.Fouvry@universite-paris-saclay.fr
\and 
Michel Waldschmidt\\
Sorbonne Universit\' e and Universit\' e de Paris, \\
CNRS, IMJ--PRG, \\
75005 Paris, France\\
E-mail: michel.waldschmidt@imj-prg.fr}

\date{\today}
\date{\misajour}

\maketitle

\renewcommand{\thefootnote}{}



\renewcommand{\thefootnote}{\arabic{footnote}}
\setcounter{footnote}{0}
\hfill  {\it Cet article est d\' edi\' e \`a Henryk Iwaniec,}

\hfill{\it  en t\' emoignage de notre admiration }
 
 \hfill{\it  pour ses contributions majeures }
 
 \hfill{\it \`a la th\' eorie analytique des nombres.}
     
\begin{abstract}  
We consider some families of binary binomial forms $aX^d+bY^d$, with $a$ and $b$ integers. Under suitable assumptions, we prove that every rational integer  $m$ with  $|m|\ge 2$ is
only represented  by a finite number of the forms of this family (with varying  $d,a,b$).  Furthermore 
{the number of such forms of degree $\ge d_0$ representing $m$  is bounded by   $O(|m|^{(1/d_0)+\epsilon})$} uniformly for $\vert m \vert \geq 2$.  We also prove that the integers in the interval $[-N,N]$ represented by one of the form of the family with degree $d\geq d_0$ are almost all represented by some form of the family with degree $d=d_0$. 

In  a previous {paper} we investigated the particular case where the binary binomial forms are positive definite. We now treat the general case by using  a lower bound for linear forms of logarithms.    \end{abstract}

\section{Introduction} \label{intro}
 When $d$, $a$ and $b$ are rational integers different from $0$, with $d\geq 3$, Theorem 1.1 of \cite{SX} gives 
 an asymptotic estimate for  the number of rational integers in the interval $[-N,N]$ represented by the binary form 
 $aX^d+bY^d$. This {estimate} has the shape
 $$
 C_{a,b}N^{2/d} +O(N^\beta)\quad\hbox{ {as} $N\to\infty$},
 $$
 where the exponent $\beta <2/d$ is explicit and where the constant $C_{a,b} >0$ is also explicit (it corresponds to the constant $C_F=A_FW_F$ in  \cite[Corollary 1.3]{SX} associated with the binary form $F(X,Y)=F_{a,b,d} (X,Y)=aX^d+bY^d$. For more precision see \S \ref{IsoAuto} below). Here we consider the representation of integers by some element of families of such binary binomial forms
 
 For every integer $d\ge 3$, let $\calE_d$ be a finite subset of  $\Z\setminus\{0\}\times\Z\setminus\{0\}$ and let $\calF_d$
 be the set of binary binomial forms  $F_{a,b,d} (X,Y)$ with  $(a,b)\in\calE_d$. We are interested in the representation of integers $m \in \Z$ by some form of the family $\calF=\bigcup_{d\ge 3}\calF_d$. For $d\ge 3$ and $m$ in $\Z$, we introduce the two sets
$$
\begin{aligned} 
\calG_{\ge d}(m)=
\Big\{(&d',a,b,x,y)\;
 \mid \;
m=ax^{d'}+by^{d'}\; \hbox{ with } 
\\
& d'\ge d,\;  (a,b)\in\calE_{d'}, \; (x,y)\in\Z^2 \; \hbox{ and } 
   \max\{|x|,|y|\}\ge 2 \Big\} 
  \end{aligned}
  $$ 
 and  
$$
\calR_{\ge d}=\left\{m\in\Z\; \mid\; \calG_{\ge d}(m)\not=\emptyset \right\}.
$$
For $N$ positive integer, we denote 
 $$
 \calR_{\ge d}(N)=\calR_{\ge d}\cap[-N,N].
 $$
The set $\mathcal E_d$ must also satisfy the following hypotheses{\it 
\vskip .2cm
 (C1) : For every $(a,b)\neq (a',b')\in \calE_d$, at least one of ratios $a/a'$ and $b/b'$ is not the $d$--th power of a rational number,
\vskip .2cm
  (C2):  For every $(a,b)\neq (a',b')\in \calE_d$, at least one of ratios $a/b'$ and $b/a'$ is not the $d$--th power of a rational number.
}

The exponent $\vartheta_d<2/d$ is defined in \cite[(2.1)]{FW2}:
  \begin{equation*} 
  \theta_d=
  \begin{cases}
  \frac {24\sqrt 3+73}{60\sqrt 3 +73}=\frac {2628\sqrt 3-1009}{5471}=0.6475\dots &\text{ {for} } d=3,\\
  \frac {2\sqrt d +9}{4 d \sqrt d -6\sqrt d +9} & \text{ {for} } 4\leq d\leq 20,\\
  \frac 1{d-1} & \text{ {for} } d\geq 21.
  \end{cases}
  \end{equation*}
 When the family $\calF$ is given and when the integer  $d$ is  $\ge 3$, 
the integer  $d^\dag$ is given by the formula
$$
d^\dag :=
\begin{cases}
\inf\{ d' : d' >d \text{ such that } \calF_{d'} \ne \emptyset\}
&\hbox{if there exists  $d'>d$ such that  $\calF_{d'} \ne \emptyset$,}
\\
\infty
& \hbox{if  $\calF_{d'} = \emptyset$ for all  $d'>d$.}
\end{cases}
$$   
Our first result is the following.
\begin{theorem}[Positive definite case]\label{Th:casDefiniPositif}
Let $\calE_d\subset \Z_{>0}\times\Z_{>0}$ satisfying \it{(C1)} and \it{(C2)}  above and the equality  $\calE_d=\emptyset$ for odd $d$.  Furthermore, we suppose that 
\begin{equation}\label{Equation:CardEd}
\frac  1 d \log(\sharp\calE_d+1)\to 0\quad \hbox{ as  } \quad d\to\infty.
\end{equation}
Then 
\\
(a) For all  $m\in\Z\setminus\{0,1\}$ and all  $d\ge 4$, the set  $\calG_{\ge d}(m)$ is finite. Furthermore, for all  $d\ge 4$ and all  $\epsilon>0$, we have, as  $|m|\to\infty$, 
$$
\sharp \calG_{\ge d}(m)=O\left(|m|^{(1/d)+\epsilon  }\right).
$$
(b)
For all 
$d\ge 4$ and all $\epsilon>0$, we have, as  $N\to\infty$,
$$
\sharp \calR_{\ge d}(N)=
\left(\sum_{(a,b)\in\calE_d}  C_{a,b}\right) N^{2/d}+ O_{d,\epsilon} \Bigl(N^{\max\{\vartheta_d+\epsilon,2/d^\dag\}}\Bigr).
$$
(c) In the above formula, we have $C_{a,b} = A_{F_{a,b,d}} W_{F_{a,b,d}}$,
where 
 $$
A_{F_{a,b,d}} = \iint_{\vert ax^d+by^d\vert \leq 1} dxdy
$$
and where the values of the rational positive numbers $W_{F_{a,b,d}}$ are given in Proposition \ref{valuesofWF}. 
\end{theorem}
\begin{remark} One finds explicit values, in terms of the $\Gamma$--function, of the fundamental area $A_{F_{a,b,d}}$ in \cite[Corollary 1.3]{SX}.
 
The hypothesis  $\sharp \calE_d\le d^{A_1}$ in  \cite[Theorem 1.13]{FW2}, which implies the condition  $(\mathrm {iii})$  in  Definition \ref{agreable} below  of a	 $(A,A_1, d_0, d_1, \kappa)$--regular family,  
is replaced here   by  \eqref{Equation:CardEd} which cannot be omitted: for  $d\ge 3$ and  $N=2^{2d}+1$, each of $2^d$ integers of the form $a2^d+1$, $a=1,2,3,\dots, 2^d$ is  represented by one of the form  $aX^d+Y^d$ with the choice $x=2$, $y=1$. 
\end{remark}
 Our second result is
\begin{theorem}[General case]\label{Th:casGeneral}
There exists an absolute constant  $\eta>0$ with the following property.
We suppose that there exists  $d_0>0$ such that, for all    $d\ge d_0$, we have the inequality
$$
\max_{(a,b)\in\calE_d}\{|a|,|b|\}\le \exp(\eta d/\log d).
$$  
Then
\\
(a) For all  $m\in\Z\setminus\{-1,0,1\}$ and all  $d\ge 3$, the set  $\calG_{\ge d}(m)$ is finite. Furthermore,    for all  $d\ge 3$ and all $\epsilon>0$, we have, as $|m|\to\infty$,
$$
\sharp \calG_{\ge d}(m)=O\left(|m|^{(1/d)+\epsilon}\right).
$$
(b) 
For all $d\ge 3$ and all  $\epsilon>0$, we have, as  $N\to\infty$,
$$
\sharp \calR_{\ge d}(N)=
\left(\sum_{(a,b)\in\calE_d}  C_{a,b}\right) N^{2/d}+ O_{d,\epsilon}  \Bigl(N^{\max\{\vartheta_d+\epsilon,2/d^\dag\}} \Bigr).
$$ 
(c) The properties of the constant $C_{a,b}$ are the same as in Theorem \ref{Th:casDefiniPositif} (c).
\end{theorem}
   
  We will prove the result with the choice
$$
     \eta=2^{-81}3^{-16},
$$
 corresponding to the right--hand side  of   \eqref{Equation:mu} for $\lambda=3$. 
 
 In both Theorems \ref{Th:casDefiniPositif} and  \ref{Th:casGeneral}, the proof of the bound  for  $
\sharp \calG_{\ge d}(m)
$ is based on the explicit estimate  \eqref{Equation:majorationRdm}. The fact that  $ \calG_{\ge d}(m)$ is finite for all $m\not\in\{-1,0,1\}$ is not a consequence of the bound for $\sharp \calR_{\ge d}(N)$ (see Example \ref{Exemple}). 

Compared to \cite{FW2}, our new tool is a lower bound for linear forms in logarithms; the finiteness of the number of representations of a given integer $m$ depends of this estimate. As we will show in section  \ref{Section:Conjectures}, the $abc$ Conjecture would give an estimate very close to what would be deduced from conjectures on linear forms in logarithms.
 
 \section{A new definition of a regular family}\label{section:reguliere}

\begin{definition}\label{Definition:reguliere}
We will say that an infinite set  $\calF$ of binary forms  with discriminants different from zero and with degrees  $\ge 3$ is  {\em regular } if there exists a positive integer  $A$    satisfying the following two conditions 
\\ 
$(\mathrm {i})$
Two forms of the family   $\calF$ are ${\rm GL}(2,\Q)$--isomorphic  if and only if they are equal,\\
$(\mathrm {ii})$
For all  $\epsilon>0$, there exist two positive integers  $N_0=N_0(\epsilon)$ and  $d_0=d_0(\epsilon)$ such that, for all  $N\ge N_0$, the number of integers   $m$ in the interval $[-N,N]$ for which there exists $d\in\Z$, $(x,y) \in \Z^2$ and $F\in \calF_d$ satisfying 
$$
d\ge d_0,\quad
 \max\{\vert x\vert, \vert y \vert\} \geq A\quad\hbox{and }\quad F(x,y)=m
$$ 
is bounded by  $N^\epsilon$. 
\end{definition}
Of course the symbol $\mathcal F_d$ above, denotes the subset of $ \mathcal F$ of forms with degree $d$.
 
For the truth of Theorem  \ref{Th:VarianteTheorem1.11} below, one cannot drop the parameter $A$, as one sees by considering the family of cyclotomic forms  \cite{FW1} where the  hypothesis  $(\mathrm {ii})$ is satisfied with $A=2$ but not with  $A=1$.

Recall the  definition 1.10 of a  $(A,A_1, d_0, d_1, \kappa)$--regular family as it is introduced in \cite{FW2}.

\begin{definition}\label{agreable} Let $A$, $A_1$, $d_0$, $d_1$ be integers and let $\kappa$ be a real number such that
$$ 
A\geq 1, \, A_1\geq 1, \,d_1\geq  d_0 \geq 0, \, 0< \kappa <A.
$$ 
 Let $\calF$ be a set of binary forms with integral coefficients and with discriminants different from zero.  We say that $\calF$ is {\em$(A,A_1, d_0, d_1, \kappa)$--regular} if it satisfies the following conditions:
 \begin{enumerate}[label=\upshape(\roman*), leftmargin=*, widest=iii]
\item
\label{it:1}
The set $\calF$ is infinite,
\item 
\label{it:2}
All the forms of $\calF$ have their degrees $\geq 3$,
\item 
 \label{it:3}
For all $d\geq 3$, we have  the inequality $\sharp \calF_d  \leq d^{A_1} $,  
\item 
\label{it:4}
Two forms of $\calF$ are isomorphic if and only if they are equal,
\item 
\label{it:5}
For any  $d\ge \max\{d_1,d_0+1\}$, the following holds
$$
\left. 
\begin{array}{ll}
F\in \calF_d ,   \\
(x,y) \in \Z^2 \text{ and }  F(x,y) \ne 0, \\
\max\{\vert x\vert, \vert y \vert\} \geq A,
\end{array}
\right\}
 \Rightarrow \max\{\vert x \vert, \vert y \vert \}\leq \kappa \left\vert F(x,y)\right\vert^{\frac 1{d-d_0}}. 
$$ 
\end{enumerate}
\end{definition}
These two definitions are not independent since we have
\begin{lemma}\label{lemma:ComparaisonRegularite}
If a family of binary forms is $(A,A_1, d_0, d_1, \kappa)$--regular  in the sense of Definition \ref{agreable}  then it is
also regular in the sense of Definition \ref{Definition:reguliere}.
\end{lemma}
 
\begin{proof}
Suppose that the family $\calF$ satisfies the condition  $(\mathrm {v})$ of Definition  \ref{agreable}. Let $\epsilon>0$, let  $N_0$ sufficiently large  and let $d_2>2/\epsilon$. We use $d_0$ and $d_1$ as in the Definition \ref{agreable}  and we replace  $d_0$ by  $\max\{d_1,d_0+1\}+d_2$ in the condition $(\mathrm {ii})$ of the Definition \ref{Definition:reguliere}. Let  $N\in\Z$, $d\in\Z$,  $m\in\Z$,  $(x,y) \in \Z^2$ and  $F\in \calF_d$ such that
$$
N\ge N_0,\quad
d\ge \max\{d_1,d_0+1\}+d_2,\quad  |m|\le N,\quad
  X\geq A\quad\hbox{and}\quad F(x,y)=m
$$ 
with  $X:=\max\{\vert x\vert, \vert y \vert\} $. 
From condition  $(\mathrm {v})$ in  Definition \ref{agreable},  we deduce 
$$
A^{d-d_0}\le X^{d-d_0}\le \kappa^{d-d_0}|m|\le  \kappa^{d-d_0}N.
$$
From these inequalities we deduce on the one hand
$$
(d-d_0)\log (A/\kappa)\le \log N,
$$
which is 
$$
d\le d_0+\frac{ \log N}{\log (A/\kappa)},
$$
and on the other hand
$$
X\le \kappa N^{1/(d-d_0)}\le \kappa N^{1/d_2}.
$$
The condition  $(\mathrm {iii})$ of the Definition  \ref{agreable}  states  that family  $\calF$ contains at most  $d^{A_1}$ forms with degree $d$.
One deduces that the number of  $(d, x,y,F)$ (such that $F(x,y)=m$ with degree $F$ equal  to $d$) and also the number of  $m$,   are  bounded by  $O(N^{2/d_2}(\log N)^{A_1+1})$. 
\end{proof}
 
\begin{example}\label{Exemple}
Let $(\ell_d)_{d\ge 3}$ be a sequence of positive integers. Let  $\calF$  be the family obtained by considering the sequence of binary forms   
 $F_d(X,Y)=(X-dY)^{2d}+\ell_dY^{2d}$. 
 We have the equalities  
 \begin{equation}\label{397}
 F_d(d,1)=\ell_d\quad\hbox{and}\quad  F_d(d-1,1)=F_d(d+1,1)=\ell_d+1.
 \end{equation}
 We then check that this family is regular, in the meaning of  Definition  \ref{Definition:reguliere} if and only if, when  $N$ tends to infinity, we have 
$$
\frac 1 {\log N} \log  \sharp
\Bigl\{
\{\ell_d \mid d\ge 3\}\cap [1,N]\Bigr\}
\to 0.
$$
Choosing  $(\ell_d)_{d\ge 3}$ to be the sequence
$(1,2,4,\; 1,2,4,8,\; 1,2,4,8,16,\; \dots)$
defined by the formula
$$
\ell_d=2^j\quad \hbox{ when } d=\frac {k(k+1)}2+j=1+2+\cdots+k+j,\quad k\ge 2, \quad 0\le j\le k,
$$
we obtain an example of a regular family   for which there exists an infinite set of integers $m$  with infinitely many representations under the form    $m=F_d(x,y)$. 

The family $\mathcal F$  is  regular in the meaning of   Definition \ref{agreable}  only if  
\begin{equation}\label{417}
\ell_d\ge (d/\kappa)^{d-d_0}
\end{equation}
for all   $d\ge \max\{d_1,d_0+1\}$. This follows from \eqref{397}. For instance the condition \eqref{417}  is not satisfied when the sequence $(\ell_d)_{d\ge 2}$ is bounded.
\end{example}
We now turn our attention to the statement of \cite[Theorem 1.11]{FW2} when one considers a family which satisfies the new notion of regularity.
The conclusion of Theorem 1.11 of our text  \cite{FW2} remains true when one replaces the assumption that the family is   $(A,A_1, d_0, d_1, \kappa)$--regular by the assumption that the family is regular in the meaning of   Definition \ref{Definition:reguliere}. Namely:

\begin{theorem}\label{Th:VarianteTheorem1.11}
Let $\calF$ be a  regular family of binary forms in the meaning of   Definition \ref{Definition:reguliere}. 
 Then for every $d\geq  3$  and every positive $\varepsilon$, the quantity
 $$
 \begin{aligned}
{\calR}_{\geq d}  \left (\calF, N, A\right) := \sharp\, \bigl\{ m: 0\leq \vert m \vert \leq N, \, \text{ there is }   \, F\in \calF \text{ with } \deg F \geq d
\\
\text{ and } (x,y)\in \Z^2 \text{ with } \max \{\vert x \vert, \vert y \vert\} \geq A, \text{ such that } F(x,y) =m
\bigr\}. 
\end{aligned}
$$
satisfies
$$
\calR_{\geq d} (\calF, N, A) = \left( \sum_{F\in \calF_d} A_F W_F \right)   \cdot N^{2/d} + O_{\calF,A,d, \varepsilon}\bigl( N^{\vartheta_d+\varepsilon} \bigr)  +   O_{\calF,A,d} \bigl( N^{2/d^\dag}\bigr), 
$$
uniformly as $N  \to\infty$.
\end{theorem}

We do not need the assumption $d\ge d_1$ which occurred in \cite[Theorem1.11]{FW2}. 
Notice that if a family does not satisfy the condition  $(\mathrm {ii})$ of the Definition  \ref{Definition:reguliere}, 
then it does not satisfy the conclusion of  Theorem \ref{Th:VarianteTheorem1.11} --- the condition $(\mathrm {ii})$ in Definition \ref{Definition:reguliere}  is essentially optimal. 

\begin{proof}[Proof of Theorem  \ref{Th:VarianteTheorem1.11}]
We use the notation
$$
\calZ_A = \Z^2\setminus \bigl( [-A, A]\times [-A,A]\bigr)
$$
introduced in \cite{FW2}.
The conditions   $(\mathrm {iii})$ and   $(\mathrm {v})$ of the Definition   \ref{agreable}  appear in \cite{FW2} when considering (3.5) and  (3.7) to show that the cardinality of the set  
$$
 \left\{ (n, F, x, y): n > d^\dag +d_0, F \in \calF_n, (x,y) \in \calZ_A,  \vert F(x,y) \vert \leq B\right\}
 $$
 is bounded by $
 o_\calF \left( B^{   2/{d^\dag}}\right)$.  Firstly we remark that it suffices to bound the cardinality of the set  $$
 \begin{aligned}
 \left\{m: 0\le m\le B, \right. \; &\hbox{ there exists } (n, F, x, y), \; 
 \\
 &
 \left.
 n > d^\dag +d_0, F \in \calF_n, (x,y) \in \calZ_A,  \vert F(x,y) \vert=m\right\}.
 \end{aligned}
 $$
 The claimed bound immediately follows  from the assumption  $(\mathrm {ii})$ of the Definition \ref{Definition:reguliere}. 
\end{proof}

The following lemma is easy. It will be used several times

\begin{lemma}\label{Lemme:Facile}
Let  $\theta>0$. 
We suppose that there exists $d_0\ge 3$ such that,     for  $m$ and  $d$ in  $\Z$ with   $|m|\ge 2$ and   $d\ge d_0$, the conditions $$
\hbox{
$d'\ge d$, $(a,b)\in \calE_{d'}$, $\max\{|x|,|y|\}\ge 2$  and  $m=ax^{d'}+by^{d'}$ }
$$
imply the inequality
$$
X^{d'}\le |m|^\theta
$$
with  $X:=\max\{|x|,|y|\}$. 
We also suppose that the condition  \eqref{Equation:CardEd} is satisfied. Then 
\\
(a) For every $m\in\Z\setminus\{-1,0,1\}$ and every $d\ge 3$, the set  $\calG_{\ge d}(m)$ is finite.  In addition  for
every $d\ge d_0$ and every $\epsilon>0$, we have, as $|m|\to\infty$, the upper bound 
$$
\sharp \calG_{\ge d}(m)=O(|m|^{(\theta +\epsilon)/d}).
$$ 
(b) For every  $d\ge d_0$ and every  $\epsilon>0$, there exists $N_0$ such that, for    $N\ge N_0$, we have the inequality
$$
\sharp \calR_{\ge d}(N)\le N^{(2\theta+\epsilon)/d}.
$$
\end{lemma}

\begin{proof}
Let  $|m|\ge 2$, $d\ge d_0$ and $d'\ge d$. 
The inequalities 
$$
2^{d'}\le X^{d'}\le |m|^{\theta}
$$
imply 
$$
d'\le \frac{\theta \log |m|}{\log 2}
\quad\hbox{and }\quad X\le |m|^{\theta/d'}.
$$
The cardinality of the set $\calG_{\ge d}(m)$  is less than 
$$
4  |m|^{\theta/d} \sum_{d'=d}^{ \lfloor \frac {\theta\log |m|}{\log 2}\rfloor} \sharp\calE_{d'},
$$ 
since, when one unknown is fixed in the equation $m = ax^{d'} +by^{d'}$, the other unknown takes two values at most.
The fact that  $\calG_{\ge d}(m)$  is a finite set for  $d\ge 3$ and  $|m|\ge 2$ follows from 
  the fact that $\bigcup_{3\le d'<d}\calF_{d'}$ is also finite. 
Thus the  assertion  (a) is a consequence of  \eqref{Equation:CardEd}. Finally the assertion  (b) follows from 
\begin{equation}\label{Equation:majorationRdm}
\sharp \calR_{\ge d}(N)\le 4  N^{2\theta/d} \sum_{d'=d}^{ \lfloor \frac {\theta\log N}{\log 2}\rfloor} \sharp\calE_{d'}.
\end{equation}
\end{proof}
\bigskip

  \begin{proof}[Proof of Theorem \ref{Th:casDefiniPositif}]
  The equality  $ax^d+by^d=m$ with  $a$ and  $b>0$  and  $d\geq 4$ even  implies the inequality  $X^d \le m$.  Lemma  \ref{Lemme:Facile} applied with $\theta =1$ proves the part  (a) of Theorem \ref{Th:casDefiniPositif}. We also check  the  condition $(\mathrm {ii})$ in the Definition \ref{Definition:reguliere}  of a regular family for the value $A=2$.  To prove  the assertion  (b) it remains to apply  Theorem \ref{Th:VarianteTheorem1.11} since the item (i) of Definition  \ref{Definition:reguliere} is fulfilled by Corollary  \ref{nonisomorphic} below.
    \end{proof}

  \section{Isomorphisms between binomial binary forms and their automorphisms}\label{IsoAuto}We recall the action of the group of matrices ${\rm GL}(2, \Q)$ on the set ${\rm Bin}(d, \Q)$ of binary forms with degree $d$, with  rational coefficients and with non zero discriminant. If $F= F(X,Y)$ and $\gamma= \begin{pmatrix} a_1 &a_2\\ a_3& a_4 \end{pmatrix}$ respectively belong to ${\rm Bin}(d, \Q)$ and ${\rm GL}(2, \Q)$, we define
  $$
 ( F\circ \gamma) (X,Y) =F(a_1X+a_2 Y, a_3 X+a_4 Y).
  $$
  By definition, we say that the two forms $F$ and $G$ are isomorphic if and only if there exists $\gamma \in {\rm GL}(2, \Q)$ such that $F\circ \gamma =G.$ 
  The group of automorphisms of a form $F$ is
  $${\rm Aut} (F, \Q) = \{ \gamma \in {\rm GL}(2, \Q) : F \circ \gamma =F \}.
  $$
  \begin{proposition}\label{isobetweenbinforms}
  Let $d\geq 3$ and $a$, $b$, $a'$ and $b'$ be integers different from zero. Then the two binary forms $aX^d+bY^d$ and $a'X^d+b'Y^d$ are isomorphic if and only if  at least one of the  following two conditions hold
  \begin{enumerate}
  \item the ratios $a/a'$ and $b/b'$ are both $d$--th powers of a rational number,
  \item the ratios $a/b'$ and $b/a'$ are both $d$--th powers of a rational number.
  \end{enumerate}
  \end{proposition}
  \begin{proof} The proof is an extension of the proof of \cite[Lemma 1.14]{FW2}
  which worked  under the restrictions  that  $d$ is an even integer and $a$, $b$, $a'$ and $b'$ are all positive. We quickly give the necessary modifications to obtain 
  Proposition \ref{isobetweenbinforms}.  Indeed the beginning of the proof of  \cite[Lemma 1.14]{FW2} does not require these restrictions. They are only used at the very last item of the proof where we appeal to them to  prove that   if  $\gamma = \begin{pmatrix} a_1 & a_2 \\ a_2 & a_4 \end{pmatrix}$ with $a_1a_2a_3a_4\ne 0$, the equality
  $$F_{a,b,d} \circ \gamma =F_{a',b',d}$$
  cannot hold. Indeed, if it holds, a computation  leads to the equalities
  $$
  \frac {a_1}{a_3}= - \frac ba \left( \frac {a_4}{a_2}\right)^{d-1}, \   \left(\frac {a_1}{a_3}\right)^2= - \frac ba \left( \frac {a_4}{a_2}\right)^{d-2}.
  $$
  Dividing  the second  equality  by the first one we obtain the equality  $a_1/a_3 =a_2/a_4$. This is impossible since $\det \gamma\ne 0$.
  \end{proof}The following corollary is straightforward
  \begin{corollary}\label{nonisomorphic} Let $\mathcal F$ be a family of binomial forms
  $(F_{a,b,d})$ with $d\geq 3$, $(a,b)\in \mathcal E_d$, where $\mathcal E_d$ satisfies the conditions (C1) and (C2)
  of \S \ref{intro}.
  Then $\mathcal F$ satisfies the item (i) of Definition \ref{Definition:reguliere} and the item (iv) of Definition  \ref{agreable}.
  \end{corollary}
  We now recall the values of the constants $W_{F_{a,b,d}}$. More generally, for any binary form $F$, the constant $W_F$ is a rational number only depending on the group ${\rm Aut} (F, \Q)$, more precisely on  lattices defined by 
  some subgroups of ${\rm Aut} (F, \Q)$. The constant $W_F$  has a rather intricate definition but in the case of the binomial forms, the corresponding group of automorphisms is rather simple (see \cite[Lemma 3.3]{SX}). The following proposition is the first part of \cite[Corollary 1.3]{SX}.

  \begin{proposition} \label{valuesofWF} 
  Let $F_{a,b,d}(X,Y) =aX^d +bY^d$ be a binary binomial form with $ab\not= 0$ and with $d \geq 3$.  Then we have
    \begin{itemize}
  \item If $a/b$ is not a $d$--th power of a rational number, then
  $$W_{F_{a,b,d}}=
  \begin{cases}
 1 & \text {if $d$ is odd},\\
  1/4 & \text {if $d$ is even}.
  \end{cases}
  $$
  \item If $a/b$ is  a $d$--th power of a rational number say $a/b = (A/B)^d$  then 
   $$W_{F_{a,b,d}}=
  \begin{cases}
  1-1/(2\vert AB\vert) & \text {if $d$ is odd},\\
 (1-1/(2\vert AB\vert))/4& \text {if $d$ is even}.
  \end{cases}
  $$
  \end{itemize}
    \end{proposition}

    \section{On the integers represented by binary binomial forms with large degree.}\label{S:OnTheIntegers}
    The following result gives an asymptotic upper bound for the number of integers represented by binary forms with high degree and for the number of representations of such integers.   
      
   \begin{theorem}\label{Th:majorationasymptotique} 
Let $d_0\geq 3$ be an integer. Let  $\lambda$ and  $\mu$ be two real numbers such that  $\lambda>2$ and  
\begin{equation}\label{Equation:mu}
 0<  \mu<2^{-81}3^{-15}\frac{\lambda-2}\lambda\cdotp
\end{equation}
We suppose that 
    $$
   \calE_d\subset \left\{(a,b)\in\Z\times\Z\,\; \mid \;  ab\not=0,\;  \max\{|a|,|b|\}\le \exp(\mu d/\log d)\right\}
    $$
 for all $d\ge d_0$.   Then 
 \\ 
(a) For every  $m\in\Z\setminus\{-1,0,1\}$ and every  $d\ge 3$, the set  $\calG_{\ge d}(m)$ is finite. Furthermore,   for every  $\epsilon>0$ and as   $|m|\to\infty$, we have  
$$
\sharp \calG_{\ge d}(m)=O\left(|m|^{\epsilon + \lambda /(2d)}\right).
$$ 
(b)
For every  $d\ge 3$, there exists  $N_0>0$ such that, for every  $N\ge N_0$, we have
$$
\sharp \calR_{\ge d}(N)\le N^{\lambda/d}.
$$    
   \end{theorem}

The set of  $m$ that we are considering in the assertion (b)  contains the set of $m$ for which the hypotheses are satisfied with  $(a,b)\in\calE_d$. By the result of  \cite{SX} (see \cite{BDW} for the particular case of binary binomial forms), each of these forms with degree  $d$ contributes to this number of  $m$ by $N^{2/d}$, up to some positive constant. The hypothesis   $\lambda>2$ is thus natural.
     
   One cannot drop the condition  $a\ne 0$. Indeed, if $a=0$, every    $m\le N$ is represented by some form   (take $y=1$, $b=m$ and $d$ sufficiently large). Similarly, one cannot drop the condition $b\ne 0$. 
 
One cannot forget the hypothesis $ \max\{|x|,|y|\}\ge 2$: take  $x=1$ and $y=-1$, then every integer  $m$ in the interval $1\le m\le N$ satisfies the equality $m=a-b$ with  $d$, $a$, $b$ satisfying the conditions of Theorem   \ref{Th:majorationasymptotique}.  
   
One cannot replace the condition $\max\{|a|,|b|\}\le \exp(\mu d/\log d)$ by \break  $\max\{|a|,|b|\}\le 2^d$, as it can be seen by the  example $x=2$, $y=a=1$, $b=m-2^d$, $m=1,\dots,2^d-1$.  In \S \ref{Section:Conjectures}  we will see  
to what extent one can hope to weaken this hypothesis by assuming either Conjecture 1 of  \cite[p. 212]{Lang}  or Conjecture $abc$. In this connection, in  
  \cite[Theorem 1.13]{FW2}, there is no hypothesis concerning  $\max \{|a|,b|\}$ when  $(a,b)$ is  in the set $\calE_d$: the only condition deals with the number of elements which must be less than $d^{A_1}$ for   \cite[Theorem 1.13]{FW2}, and must satisfy condition \eqref{Equation:CardEd} for Theorem \ref{Th:casDefiniPositif}. The example of the family $X^d+(d-2^d)Y^d$  shows that such a result cannot be extended to the case where the binary forms has real zeroes.    
  \section{A diophantine result}\label{S:ResultatDiophantien}

The central tool in the proof of  Theorem \ref{Th:majorationasymptotique} is a lower bound coming from the theory of linear forms in logarithms, more precisely \cite[Corollary 9.22]{W}.  The usual height of the rational number  $p/q$, written under its irreducible form, is defined by 
$\rmH(p/q)=\max\{|p|,q\}$ and its logarithmic height is
$$
\rmh(p/q)=\log \rmH(p/q)= \log\max\{|p|,q\}.
$$

\medskip
\begin{proposition} \label{Prop:9.22}
Let  $a_1, a_2$ be rational numbers, $b_1,b_2$ be positive integers, $A_1,A_2,B$ be real positive numbers. We suppose  for  $j=1,2$ that  
$$
B\ge \max \{\rme,\; b_1,\; b_2\},
\quad 
\log A_j\ge \max\{\rmh(a_j),\;  1\}.
$$
Then, if  $ a_1^{b_1}a_2^{b_2}\not=1$, we have the inequality 
$$ 
|a_1^{b_1}a_2^{b_2}-1|\ge \exp\left\{-C(\log B)(\log A_1)(\log A_2)\right\}
$$ 
with  $C=2^{79}3^{15}$. 
\end{proposition}
\medskip

This lower bound follows from Corollary 9.22  in  \cite[p. 308]{W} by taking 
$$
D=1, \quad  m=2, \quad \alpha_1=a_1, \quad \alpha_2= a_2
$$
and the constant  $C(m)$ defined in \cite[p. 252]{W}.

\begin{corollary}\label{Cor:axd+byd}
   Let $d$, $a$, $b$, $x$ and $y$ be rational integers. Let 
   $$
   \calA:=\max\{|a|,|b|\}, \quad X:=\max\{|x|,|y|\}.
   $$
   We suppose  $d\ge 2$, $\calA\ge 2$,  $X\ge 2$ and $ax^d+by^d\not=0$. 
   Then we have the lower bound
   $$
   |ax^d+by^d|\ge \max\{ |ax^d|, |b y^d| \} \exp\{-4C  (\log d)(\log X)(\log \calA)\}.
   $$
\end{corollary}

The conclusion is obviously false when one of the parameters $d$, $X$, $\calA$ equals $1$.

\begin{proof} 
By symmetry one can suppose that  $|ax^d|\le  |b y^d| $. 
 We use Proposition \ref{Prop:9.22} with
$$
b_1=d,\quad b_2=1,\quad a_1=\frac x y, \quad a_2=-\frac a b,
$$
$$
B=  
\begin{cases} d
 & \hbox{if $d\ge 3$}
 \\
 \rme 
 & \hbox{if $d= 2$,}
 \end{cases}
 \quad
 A_1=  
 \begin{cases} X
 & \hbox{if $X\ge 3$}
  \\
 \rme 
 & \hbox{if $X= 2$,}
 \end{cases} 
 \quad   A_2= 
 \begin{cases} \calA 
 & \hbox{if $\calA\ge 3$}
 \\
 \rme 
 & \hbox{if $\calA= 2$.}
 \end{cases}
$$
We conclude  using the inequality $1/(\log 2)^3 <4$.
\end{proof}

Corollary  \ref{Cor:axd+byd}  implies the lower bound
$$
   |ax^d+by^d|\ge X^d  \exp\{-4C  (\log d)(\log X)(\log \calA)\},
   $$
that we write as  
\begin{equation}\label{Equation:diophantienne}
   |ax^d+by^d|\ge X^{d  -4C(\log d)(\log \calA)}.
\end{equation}
 \section{Proof of Theorems  \ref{Th:majorationasymptotique} and \ref{Th:casGeneral}}

   \begin{proof}[Proof of Theorem \ref{Th:majorationasymptotique}] 
    Let  $\lambda'$ in the interval  $2<\lambda'<\lambda$ such that 
   $$ 
     \mu=\frac {\lambda'-2}{4C\lambda'},
   $$
 where $C$ is defined in Proposition \ref{Prop:9.22}.   Let  $m=ax^{d'}+by^{d'}$ with $|m|\ge 2$, $d'\ge d$, $(a,b)\in \calE_{d'}$ and  $X\ge 2$. 
    When $ax^{d'}$ and  $by^{d'}$ have the same sign,  we have  $|m|\ge X^{d'}$  and we use Lemma \ref{Lemme:Facile} with $\theta=1$.
     When $ax^{d'}$ and $by^{d'}$ have opposite signs, in order  to use Lemma  \ref{Lemme:Facile},  we can suppose that  $m\ge 2$, $a,x,y >0$  and  $b<0$. 
       
   We are firstly interested in the pairs  $(a,b)\in\bigcup_{d'\ge d} \calE_{d'}$ satisfying \break  $\max\{|a|,|b|\}=1$. By our hypotheses, we have  $a=1$ and  $b=-1$.  There is no restriction to suppose  that $d\ge \lambda'/(\lambda'-2)$, since when $m\ne 0$ is given, the equation $x^d-y^d=m$ has at most $(d-1)\cdot \sharp\{ k\mid k\vert m\}= O_d (\vert m \vert^\epsilon)$ solutions.  For $d'\ge d$, we write 
   $$
   m=x^{d'}-y^{d'}=(x-y)(x^{d'-1}+x^{d'-2}y+\cdots+xy^{d'-2}+y^{d'-1}) > X^{d'-1}\ge X^{2d'/\lambda'}.
   $$
   Thus we can use Lemma  \ref{Lemme:Facile} with $\theta=\lambda'/2$.    

 We now consider the pairs  $(a,b)\in\bigcup_{d'\ge d} \calE_{d'}$ such that $\calA:=\max\{|a|,|b|\}$ satisfies  $ \calA\ge 2$. Since we supposed that $ \calA\le \exp(\mu d'/\log d')$,  we have
   $$
     \frac {d'} {(\log d')(\log \calA)} \ge \frac 1 \mu = \frac {4C\lambda'}{\lambda'-2}\cdotp
   $$ 
Let  $X:=\max\{|x|,|y|\}$.  We deduce from  \eqref{Equation:diophantienne} the inequality
   $$
   X^{d'  -4C(\log d')(\log \calA)} \le m 
   $$
   with
   $$
   d'-4C(\log d')(\log \calA)\ge d'(1-4C\mu)= \frac {2d'}{\lambda'},
   $$
    which allows to use Lemma \ref{Lemme:Facile} with the choice $\theta=\lambda' /2$.

To conclude, we add the three values  $m=0$  and $m=\pm1$. 
\end{proof}

\begin{proof} [Proof of Theorem \ref{Th:casGeneral}] It mimics the proof of Theorem \ref{Th:casDefiniPositif} in Section \ref{section:reguliere}: 
combining 
Corollary \ref{nonisomorphic} and Theorem \ref{Th:majorationasymptotique} (b), one deduces that the family  $\calF$ satisfies the definition \ref{Definition:reguliere} of a regular family. 
\end{proof}

  \section{Conjectures}\label{Section:Conjectures}
Let $X_0$  be an integer   $\ge 2$. we introduce the following subset of  $\calR_{\ge d}$:
$$
\begin{aligned} 
\calR_{\ge d,X_0}=\Bigl\{m\in\Z\; \mid\; &
\hbox { there exists }
(d',a,b,x,y)\;
\hbox{ such that  } m=ax^{d'}+by^{d'}\; \\ \hbox{ with } 
& d'\ge d,\;  (a,b)\in\calE_{d'}, \; (x,y)\in\Z^2 \; \hbox{ and } 
   \max\{|x|,|y|\}\ge X_0 \Big\},
  \end{aligned}
  $$ 
 such that   $
\calR_{\ge d} =\calR_{\ge d,2}$.
For $N$ a positive integer   we also denote  
 $$
\calR_{\ge d,X_0}(N)=\calR_{\ge d,X_0}\cap[-N,N].
 $$
 
After the statement of Theorem \ref{Th:majorationasymptotique} in Section \ref{S:OnTheIntegers}, we gave the example of the equation $2^d-b=m$ to show that one cannot replace the condition $\max\{|a|,|b|\}\le \exp(\mu d/\log d)$ by $\max\{|a|,|b|\}\le 2^d$. Introducing a parameter $X_0$ and assuming  $\max\{|x|,|y|\}\ge X_0$, one might expect a more general result to hold. It is interesting to notice that a conjecture on lower bounds for linear forms in logarithms and the $abc$ Conjecture would produce very similar results.  
  
  \subsection{Conjecture 1 of \cite{Lang}}\label{S:ConjectureLang}
  We state   Conjecture  1 in  \cite[Introduction to Chapters X and XI, p. 212]{Lang} as follows.
      
   \begin{conjecture}
  \label{Conjecture:LW}
   Let $\epsilon>0$. There exists a constant  $C(\epsilon)>0$ only depending on  $\epsilon$ such that, if $a_1,\dots,a_n$ are rational positive numbers  and if  $b_1,\dots,b_n$ are integers, by defining 
   $$
   B_j=\max\{|b_j|, 1\}, \quad A_j=\max\{{\mathrm e}^{\rmh(a_j)}, 1\},\quad B=\max_{1\le j\le n} B_j
   $$
and by supposing that $b_1\log a_1+\cdots+b_n\log a_n\not=0$, we have the lower bound 
   $$
   |b_1\log a_1+\cdots+b_n\log a_n|> \frac {C(\epsilon)^nB}{(B_1\cdots B_nA_1^2\cdots A_n^2)^{1+\epsilon}}\cdotp
   $$
   \end{conjecture}   
   
   Actually, we will only use a weak form of this conjecture  : we will suppose the existence of a number    $\epsilon>0$ for which Conjecture \ref{Conjecture:LW}  holds. 
      
   \begin{theorem}\label{Th:ConjectureLang}
Let  $\epsilon>0$ such that Conjecture  \ref{Conjecture:LW} is verified for this value of  $\epsilon$.  Let $\lambda >2$.
Let $d_0$ be a sufficiently large integer  and let $X_0\ge 2$. We suppose 
    $$
   \calE_d\subset \left\{(a,b)\in\Z\times\Z\,\; \mid \;  ab\not=0,\;  \max\{|a|,|b|\}\le X_0^{d/d_0}\right\}
    $$
 for all $d\ge d_0$. Then for every  $d\ge d_0$, we have  
$$
\sharp \calR_{\ge d,X_0}(N)\le N^{\lambda/d}.
$$   
   \end{theorem}

\begin{proof}
Conjecture \ref{Conjecture:LW} with $n=2$, $B=B_1=B_2=d$, $A_1=X$, $A_2=\calA$ would allow to replace the conclusion of Corollary \ref{Cor:axd+byd} by 
   $$
   |ax^d+by^d|\ge X^d  \exp\{-c(\epsilon)-(1+\epsilon)\log d-(2+\epsilon)\log X-(2+\epsilon) \log \calA\}
   $$
   with $c(\epsilon)>0$ depending only on $\epsilon$. 
   Let  $\lambda'$ in the interval $2<\lambda'<\lambda$ and let    $d_0$ sufficiently large to ensure the inequality
 $$ 
  1-  \frac {2(2+\epsilon) } {d_0}    -\frac{c(\epsilon)+(1+\epsilon)\log d_0}{d_0\log 2}>\frac 2 {\lambda'}\cdotp
   $$
For   $d'\ge d$ and $m=ax^{d'}+by^{d'}$, the resulting upper bound
$$
 X^{d'}\le    |m|^{\lambda' /2}
 $$
allows to use Lemma  \ref{Lemme:Facile}. 
\end{proof}
  \subsection{Conjecture $abc$}
   
   Let  $R(m)$ be the radical of a positive integer  $m$:
   $$
   R(m)=\prod_{p \; \mathrm { prime}, \; p\mid m}p.
   $$
  The well known Conjecture  $abc$ (see for example \cite[\S 1.2]{W}) asserts that {\em for all $\epsilon>0$, there exists a constant     $\kappa(\epsilon)$ such that  if $a$, $b$, $c$ are coprime positive integers such that  $a+b=c$, then the following inequality holds
   $$
   c\le \kappa(\epsilon)  R(abc)^{1+\epsilon}.
  $$ 
}

Like in Section \ref{S:ConjectureLang},  we will only assume the existence of a number    $\epsilon>0$ for which the property holds.  

   \begin{lemma}\label{Lemme:abc1}
   Let  $\epsilon>0$ such that Conjecture  $abc$ holds.  Then under the hypotheses of Corollary \ref{Cor:axd+byd},  we have the inequality
 $$
 X^{d-2-2\epsilon}\le \kappa(\epsilon) \calA^{1+2\epsilon}|m|^{1+\epsilon}.
 $$
   \end{lemma}
   
\begin{proof}
Let $m=ax^d+by^d$.  
Without loss of generality, one can suppose $|ax^d|\ge |by^d|$. If  $|m|\ge |ax^d|$, the conclusion is obvious.  Now suppose    $|ax^d|>|m|$. After a potential change of signs, we can also suppose that 
  $a>0$, $x>0$, $y>0$  and  $b<0$.

Let  $\Delta$ be the greatest common divisor of  $ax^d$ and  $|b|y^d$ and let  $P$  be the set of prime divisors of  $\Delta$. For  $p\in P$, we write
$$
\alpha_p=v_p(a),\; \beta_p=v_p(b),\; \xi_p=v_p(x), \; \eta_p=v_p(y),\;  \delta_p=v_p(\Delta).
$$
Thus we have 
$$
\delta_p=\min\{\alpha_p+d\xi_p,\; \beta_p+d\eta_p\}.
$$
We also define 
$$
a=\tilde a \prod_{p\in P} p^{\alpha_p},\; |b|=\tilde b \prod_{p\in P} p^{\beta_p},\; x= \tilde x  \prod_{p\in P} p^{\xi_p}, \; y=\tilde y \prod_{p\in P} p^{\eta_p},
$$
so that,   for  $p\in P$, we have the equalities $v_p(\tilde a)=v_p(\tilde b)=v_p(\tilde x)=v_p(\tilde y)=0$.

Let 
$\tilde m=\Delta^{-1}m$; we have 
$$
\tilde m= \Delta^{-1}ax^d-\Delta^{-1}|b|y^d\ 
$$
with
$$
\Delta^{-1}ax^d
=\tilde a \tilde x^d\prod_{p\in P} p^{\alpha_p+d\xi_p-\delta_p}
\quad\hbox{and }\quad
\quad
\Delta^{-1}|b|y^d=
\tilde b \tilde y^d\prod_{p\in P} p^{\beta_p+d\eta_p-\delta_p}.
$$
The radical of  $\Delta$ is $\tilde\Delta:= \prod_{p\in P} p$. 
 The integers  $\Delta^{-1}ax^d $ and  $\Delta^{-1}|b|y^d$ are coprime integers, the radical of their product is less than  $\tilde\Delta \tilde a \tilde b \tilde x \tilde y $. 
We use Conjecture  $abc$ for 
$$
c=\Delta^{-1} ax^d=\tilde m +\Delta^{-1} |b|y^d.
$$
It gives the inequality
$$
\Delta^{-1} ax^d\le \kappa(\epsilon) (\tilde\Delta \tilde a \tilde b \tilde x \tilde y \tilde m)^{1+\epsilon},
$$
which is 
$$
ax^d\le \kappa(\epsilon) (a\, |b|\, xym)^{1+\epsilon}  \Delta  \prod_{p\in P} p^{(1+\epsilon)(1-\alpha_p-\beta_p-\xi_p-\eta_p-\delta_p)}.
$$ 
Since  $\delta_p\ge 1$ we have $\alpha_p+\eta_p\ge 1$, $\beta_p+\xi_p\ge 1$ and we obtain
$$
ax^d\le \kappa(\epsilon) (a\, |b|\, xym)^{1+\epsilon},
$$
that we write as
\begin{equation}\label{equation:abc1}
x^{d-1-\epsilon}\le \kappa(\epsilon) a^\epsilon (|b|\, ym)^{1+\epsilon}.
\end{equation}
We now use the bound  $|b|y^d\le ax^d$ by writing the inequality
\begin{equation}\label{equation:abc2}
y\le (a/|b|)^{1/d}x.
\end{equation} 
Then we have
$$
y^{1+\epsilon}\le (a/|b|)^{(1+\epsilon)/d}x^{1+\epsilon}
$$
and  \eqref{equation:abc1} gives 
\begin{align}
\notag
x^{d-2-2\epsilon}
&
\label{Equation:majorationxd}
\le  \kappa(\epsilon) a^\epsilon (a/|b|)^{(1+\epsilon)/d} (|b|\, m)^{1+\epsilon}
\\
&=
\kappa(\epsilon) a^{(1/d)+\epsilon+(\epsilon/d)} |b|^{1-(1/d)+\epsilon-(\epsilon/d)} m^{1+\epsilon}.
\end{align}
We again use  \eqref{equation:abc1} and  \eqref{equation:abc2} to obtain
$$
\begin{aligned}
y^{d-2-2\epsilon}&\le \kappa(\epsilon)  (a/|b|)^{1-(2/d)-2(\epsilon/d)} a^{(1/d)+\epsilon+(\epsilon/d)} |b|^{(1-(1/d)+\epsilon-(\epsilon/d))} m^{1+\epsilon}
\\
&=
\kappa(\epsilon) a^{(1-(1/d)+\epsilon-(\epsilon/d))} 
|b|^{(1/d)+\epsilon+(\epsilon/d)} m^{1+\epsilon}.
\end{aligned}
$$
Thanks to \eqref{Equation:majorationxd} we conclude the proof of Lemma \ref{Lemme:abc1}.
\end{proof}

   \begin{theorem}\label{Th:abc}
   Let  $\epsilon>0$ such that Conjecture $abc$ is verified for this value of $\epsilon$. 
   Let  $\lambda>2+2\epsilon$, let  $d_0$ be a sufficiently large integer and let  $X_0\ge 2$. We suppose 
    $$
   \calE_d\subset \left\{(a,b)\in\Z\times\Z\,\; \mid \;  ab\not=0,\;  \max\{|a|,|b|\}\le X_0^{d/d_0}\right\}
    $$
 for all  $d\ge d_0$. Then for every $d\ge d_0$, we have 
$$
\sharp \calR_{\ge d, X_0}(N)\le N^{\lambda/d}.
$$    
      \end{theorem}
      
\begin{proof} 
 
Let  $\lambda'$ in the interval  $2<\lambda'<\lambda/(1+\epsilon)$ and let $d_0$ be a sufficiently large integer such that 
the following inequality holds$$
1 -\frac{3+4\epsilon}{d_0} -\frac{\log \kappa(\epsilon)}{d_0\log 2}> \frac 2 {\lambda'}\cdotp
$$
Let   $d'\ge d$ and  $m=ax^{d'}+by^{d'}$. 
From Lemma \ref{Lemme:abc1} we deduce the bound  
$ X^{d'}\le   |m|^\theta$ with   $\theta=\lambda'(1+\epsilon)/2$, which allows to apply Lemma \ref{Lemme:Facile}. 
 \end{proof}


\end{document}